\documentclass[12pt]{amsart}

\usepackage{macros}
\standardmargins\standardpackages\standardrenews\standardmakes

\theoremlevel{}

\draftfalse

\newcommand{\bH}{\mathbb H}

\begin{document}
\title{Intersecting limit sets of Kleinian subgroups and Susskind's question}

\authortushar
\authordavid


\begin{abstract}
We construct a non-elementary Fuchsian group that admits two non-elementary subgroups with trivial intersection and whose radial limit sets intersect non-trivially. This negatively answers a question of Perry Susskind (1989) that was stated as a conjecture by James W. Anderson (2014).
\end{abstract}
\maketitle

Fix $2\leq d\leq \infty$, and let $\bH^d$ denote $d$-dimensional real hyperbolic space and
$\del\H^d$ denote its boundary; see \cite[\62]{DSU} for background regarding the case $d = \infty$. 
Elements of the isometry group of $d$-dimensional real hyperbolic space, $\Isom(\H^d)$,  come in three flavors, see \cite[\64]{BIM} or \cite[\66]{DSU}. More precisely, every $g\in\Isom(\H^d)$ is exactly one of the following three types: \textsf{elliptic} if it has bounded orbits; \textsf{parabolic} if it has unbounded orbits and it has one fixed point on the boundary, and \textsf{loxodromic} if it has unbounded orbits and it fixes two points on the boundary. 

A group $G\leq\Isom(\bH^d)$ is called \textsf{(strongly) discrete} if for every $R > 0$
\[
\#\{g\in G: \mathrm{dist}_\H(\0,g(\0))\leq R\} < \infty.
\]
If one replaced $\0 \in \H^d$ by any other point, then the class of (strongly) discrete groups would remain unchanged. The adverb ``strongly'' is used in infinite dimensions since there exist a variety of weaker notions of discreteness when one leaves the finite-dimensional realm; cf. \cite[\65]{DSU}. 
Discrete subgroups of $\Isom(\bH^d)$ are more commonly referred to as \emph{Fuchsian} groups when $d=2$, and as \emph{Kleinian} groups when $d=3$. Their study in dimension $d > 3$ is a more recent phenomenon, e.g., see Michael Kapovich's 2008 survey \cite{Kapovich}.

The \textsf{limit set} of $G$ is the set
\[
\Lambda(G) := \{\xi\in\del\bH^d: \exists (g_n)_1^\infty\text{ in $G$} ~\text{such that}~ g_n(\0)\tendsto n \xi\}.
\]
The group $G$ is called \textsf{non-elementary} if its limit set contains at least three points, in which case its limit set must be perfect and contain uncountably many points \cite[Proposition 10.5.4]{DSU}. In this case, the limit set $\Lambda(G)$ may be characterized (\cite[Proposition 7.4.1]{DSU}) as the smallest (in the sense of inclusion) closed $G$-invariant subset of $\del\bH^d$ that contains at least two points.
In the sequel we will be interested in two distinguished subsets of the limit set $\Lambda(G)$, viz. the \textsf{radial limit set}, defined as 
\[
\Lr(G) := \{ \eta \in \Lambda(G) ~:~ \exists c>0, ~\exists (g_n)_1^\infty\text{ in $G$}  ~\text{such that}~ [0, \eta]_\H \cap B(g_n(\0),c) \neq \emptyset \},
\]
and the \textsf{uniformly radial limit set}, defined as 
\[
\Lur(G) := \{ \eta \in \Lambda(G) ~:~ \exists c>0  ~\text{such that}~ [0,\eta]_\H \subset \bigcup_{g \in G} B(g(\0),c) \}.
\]
In these definitions, $[\0,\eta]_\H$ denotes the hyperbolic ray from $\0$ to $\eta$, and $B(x,r)$ the hyperbolic ball of radius $r$ centered at $x \in \H$. We note that if one replaced $\0 \in \H^d$ by any other point, then the definitions of these sets  would remain unchanged. Further, it follows directly from the definitions that 
\[
\Lur(G) \subset \Lr(G) \subset \Lambda(G).
\]
To glean some dynamical intuition for these sets, project $[\0,\eta]_\H$ onto the associated hyperbolic orbifold $M := \H^d/G$ where it maps to a geodesic ray starting from the point on $M$ which corresponds to the origin $\0 \in \H$. For $\eta \in \Lr(G)$ this ray performs a recurrent geodesic excursion on $M$, i.e. there exists a bounded region in $M$ which gets visited infinitely often; whereas for $\eta \in \Lur(G)$ the ray describes a bounded excursion, i.e. each point on the ray is at most a bounded distance away from the starting point in $M$. 

Since the work of Bernard Maskit in the 1970s, there has been sustained interest in developing results that elucidate the relationship between (the limit set of) the intersection of a pair of subgroups of a Kleinian group, and the intersection of their respective limit sets, see \cite{Maskit2, Susskind, Susskind2, SusskindSwarup, Anderson6, Anderson7, Anderson4,Yang2}. A natural question, which arose amidst these results, was to clarify the extent to which a non-empty intersection of the limit sets of a pair of subgroups was related to the intersection of the respective subgroups being non-trivial.
In such a vein, given a pair of non-elementary subgroups $G_1, G_2$ of a non-elementary Kleinian group $G \leq \Isom(\H^d)$, Perry Susskind asked in \cite[\66.3]{Susskind2}  (for $d = 3$) whether 
\begin{equation}
\label{conjecture}
\Lr(G_1) \cap \Lr(G_2) \subseteq \Lambda(G_1 \cap G_2).
\end{equation}
Recently James W. Anderson stated the higher-dimensional case (for $2\leq d < \infty$) of formula \eqref{conjecture} as a conjecture (attributed there to Susskind) in \cite[Conjecture $0$]{Anderson5}. To demonstrate the sharpness of \eqref{conjecture}, Anderson \cite[\63]{Anderson5} constructed examples in $\H^2$ to show that neither $\Lr(G_1) \cap \Lr(G_2) \subseteq \Lr(G_1 \cap G_2)$ or $\Lambda(G_1) \cap \Lambda(G_2) \subseteq \Lambda(G_1 \cap G_2)$ are true for arbitrary non-elementary subgroups $G_1, G_2$ of a fixed $G \leq \Isom(\H^d)$. He went on to prove \cite[Proposition 1]{Anderson5} the following:
\begin{theorem}[Anderson]
Let $G_1, G_2$ be non-elementary subgroups of a non-elementary Kleinian group $G \leq \Isom(\H^d)$, for some $2 \leq d < \infty$. Then 
\[
\Lr(G_1) \cap \Lur(G_2) \subseteq \Lr(G_1 \cap G_2).
\]
\end{theorem}

\begin{remark}
\label{remark2}
Anderson had assumed that $G$ was {\it purely loxodromic} (i.e.~one that did not contain parabolic or elliptic elements), which is an assumption that was not used in Anderson's proof. One could generalize Anderson's theorem to the setting of proper geodesic Gromov hyperbolic metric spaces, since the proof does not make use of any facts about real finite-dimensional hyperbolic space in any essential way.
\end{remark}

The main result of this note proves that despite Anderson's positive result, \eqref{conjecture} is too optimistic.

\begin{theorem}
\label{maintheorem}
There exists a Fuchsian group $G \leq \Isom(\H^2)$ that admits subgroups $G_1,G_2 \leq G$ such that $\Lr(G_1)\cap \Lr(G_2) \neq \emptyset$, but $G_1\cap G_2 = \lb e\rb$. In particular, \eqref{conjecture} is false.
\end{theorem}
\begin{proof}
Let $G$ be a Schottky group, \cite[\610]{DSU}, generated by two loxodromic elements $a,b\in\Isom(\H^2)$ chosen so that $G$ will be purely loxodromic.\footnote{This is the classical construction of a Schottky group that is free on $2$ generators (see e.g. \cite{Maskit}): consider two pairs of mutually disjoint circles $C_a,C_a'$ and $C_b,C_b'$ that intersect $\del\H^2$ orthogonally and have mutually disjoint interiors. Now choose a (loxodromic) isometry $a \in \Isom(\H^2)$ that maps $C_a$ to its partner $C_a'$ and sends the exterior of $C_a$ onto the interior of its partner; and similarly choose another loxodromic element $b \in \Isom(\H^2)$ that analogously acts on the pair $C_b,C_b'$.} Let $E = \{a,b\}$, and let $E^*$ be the set of all finite words using the alphabet $E$. Let $I\subset E^*$ be an infinite set with the property that no element of $I$ is an initial segment of any other element of $I$, for example $I = \{b^n a : n\in\N\}$. Let $(\omega_n)_1^\infty$ be an indexing of $I$, and for each $n$ let $\tau_n$ denote the word resulting from writing $\omega_n$ backwards, e.g. $\omega_n = b^n a$ and $\tau_n = a b^n$. Next, for each $n$ let
\[
\theta_n = \omega_1 \tau_1 \cdots \omega_n \tau_n.
\]
Let $G_1 = \lb \theta_n : \text{$n$ odd}\rb$ and $G_2 = \lb \theta_n : \text{$n$ even}\rb$, where $\lb S \rb$ denotes the group generated by a set $S$. Since $G$ is a Schottky group, the embedding $G\ni g \mapsto g(\0) \in \H^2$ is a quasi-isometry, and thus the sequence $(\theta_n(\0))_1^\infty$ converges to some element $\eta \in \del\H^2$, for which there exists $c>0$ such that for all $n$
\[
[0, \eta]_\H \cap B(\theta_n(\0),c) \neq \emptyset,
\]
see e.g. \cite[Lemma 10.4.4]{DSU} for details. It follows that $\eta \in \Lr(G_1)\cap \Lr(G_2)$, and in particular $\Lr(G_1)\cap \Lr(G_2) \neq \emptyset$.

So to complete the proof, we must show that $G_1 \cap G_2 = \lb e\rb$. It suffices to show that the sequence $(\theta_n)_1^\infty$ generates the free group on infinitely many elements, or equivalently that $(\omega_n \tau_n)_1^\infty$ generates the free group on infinitely many elements. So consider any non-trivial reduced word in the symbols $(\omega_n\tau_n)_1^\infty$, say
\begin{equation}
\label{word}
(\omega_{n_1}\tau_{n_1})^{\epsilon_1} \cdots (\omega_{n_m} \tau_{n_m})^{\epsilon_m},
\end{equation}
where $\epsilon_i \in \{\pm 1\}$ for each $i = 1,\ldots,m$. 
To find the reduced form of this word, we will first subdivide it into words of the form $\omega_n^\pm$ and $\tau_n^\pm$, group them in pairs such that the first and last of these words stand alone, perform the reduction within each pair, and then show that no further reduction occurs within the entire word.

Consider any $i$ such that $\epsilon_i = 1$ and $\epsilon_{i + 1} = -1$. Since the word \eqref{word} is reduced with respect to the symbols $(\omega_n\tau_n)_1^\infty$, we have $n_{i + 1} \neq n_i$, and so the condition on $I$ guarantees that when we reduce the word $\tau_{n_i} \tau_{n_{i + 1}}^{-1}$, the leftmost and rightmost letters are unaffected. Similarly, if $\epsilon_i = -1$ and $\epsilon_{i + 1} = 1$, then when we reduce the word $\omega_{n_i}^{-1} \omega_{n_{i + 1}}$, the leftmost and rightmost letters are unaffected. If $\epsilon_i = \epsilon_{i + 1}$, then the corresponding substring of \eqref{word} does not require any reduction at all since it contains only positive powers of the letters $a,b$. In conclusion, when we reduce \eqref{word} (with respect to the symbols $a,b$), the first and last substrings of the form $\omega_n^{\pm}$ or $\tau_n^\pm$ are unaffected, and every other pair of such substrings is reduced together but the leftmost and rightmost letters of such a pair are unaffected, meaning that no further reduction will be done. Thus the word \eqref{word} does not reduce to the empty word. Since the word \eqref{word} was arbitrary, this means that $(\omega_n \tau_n)_1^\infty$ generates the free group and thus $G_1\cap G_2 = \lb e\rb$.
\end{proof}

\begin{remark}
Theorem \ref{maintheorem} can be generalized to the setting of Gromov hyperbolic metric spaces that admit a purely loxodromic Schottky group. 
Such spaces include real, complex and quaternionic hyperbolic spaces, referred to as {\it algebraic hyperbolic spaces} in \cite[\61.1.1]{DSU}.
\end{remark}

\begin{remark}
In our proof of Theorem \ref{maintheorem} , notice that $\#(\Lr(G_1)\cap \Lr(G_2))\geq1$. This provokes the question of how large $\Lr(G_1)\cap \Lr(G_2)$ can be when $G_1 \cap G_2 = \lb e\rb$. For instance, if $\Lr(G_1)\cap \Lr(G_2)$ were assumed to have positive Hausdorff dimension, would it then follow that \eqref{conjecture} holds?
\end{remark}

\noindent {\bf Acknowledgements.} The first-named author was supported in part by a 2017-2018 Faculty Research Grant from the University of Wisconsin-La Crosse. The second-named author was supported by the EPSRC Programme Grant EP/J018260/1. The authors thank James Anderson, Perry Susskind and an anonymous referee for helpful comments.

\bibliographystyle{amsplain}

\bibliography{bibliography}

\end{document}